\documentclass[11pt, a4paper]{amsart}
\usepackage{a4wide}
\usepackage{graphicx,enumitem}
\usepackage{subfigure}
\usepackage{units}
\usepackage{color}

\usepackage{latexsym}

\usepackage[pdftex,
            pdftitle={Covariance structure of parabolic SPDEs},
            pdfauthor={A. Lang, S. Larsson, and Ch. Schwab},
            bookmarksopen,
            colorlinks,
            linkcolor=black,
            urlcolor=black,
	   			 	citecolor=black
]{hyperref}

\allowdisplaybreaks
\numberwithin{equation}{section}

\newcommand{\R}{\mathbb{R}}
\newcommand{\N}{\mathbb{N}}

\newcommand{\IL}{\mathbb{L}}
\newcommand{\IT}{\mathbb{T}}
\newcommand{\IM}{\mathbb{M}}

\newcommand{\ga}{\alpha}

\renewcommand{\gg}{\gamma}

\newcommand{\go}{\omega}

\newcommand{\gO}{\Omega}

\newcommand{\cA}{\mathcal{A}}
\newcommand{\cB}{\mathcal{B}}

\newcommand{\cD}{\mathcal{D}} 

\newcommand{\cF}{\mathcal{F}}

\newcommand{\cH}{\mathcal{H}}

\newcommand{\cP}{\mathcal{P}}

\newcommand{\cX}{\mathcal{X}}
\newcommand{\cY}{\mathcal{Y}}

\newcommand{\LHS}{L_{\mathrm{HS}\,}}

\newcommand{\LNP}{L^+_{\mathrm{N}\,}}

\DeclareMathOperator{\E}{\mathbb{E}} 
\DeclareMathOperator{\Cov}{\mathsf{Cov}}

\newcommand{\dd}{{\mathrm d}}

\newtheorem{lemma}{Lemma}[section]

\newtheorem{theorem}[lemma]{Theorem}

\newtheorem{corollary}[lemma]{Corollary}
\theoremstyle{remark}

\theoremstyle{definition}

 \newcommand*{\lrscript}[5]{{\vphantom{#1}}_{#2}^{#3}{#1}_{#4}^{#5}}
\newcommand*{\dualpair}[4]{\ensuremath{\lrscript{\langle}{#1}{}{}{} #3 , #4 \lrscript{\rangle}{}{}{#2}{}}}
\DeclareMathOperator{\Tr}{Tr}
\numberwithin{equation}{section}

\begin{document}
\title[Covariance structure of parabolic SPDEs]{
Covariance structure of parabolic stochastic partial differential equations}

\author[A.~Lang]{Annika Lang} \address[Annika Lang]{\newline 
      Seminar f\"ur Angewandte Mathematik 
\newline ETH Z\"urich \newline R\"amistrasse 101, CH--8092 Z\"urich, Switzerland.
} \email[]{annika.lang@math.ethz.ch}

\author[S.~Larsson]{Stig Larsson} \address[Stig Larsson]{\newline
  Department of Mathematical Sciences
\newline Chalmers University of Technology \newline 
and University of Gothenburg\newline
SE--412 96 Gothenburg, Sweden.} \email[]{stig@chalmers.se}

\author[Ch.~Schwab]{Christoph Schwab} \address[Christoph Schwab]{
\newline 
Seminar f\"ur Angewandte Mathematik 
\newline
ETH Z\"urich \newline R\"amistrasse 101, CH--8092 Z\"urich, Switzerland.} 
\email[]{schwab@math.ethz.ch}

\thanks{Acknowledgement. 
The work of Annika Lang and Christoph Schwab was supported in part by
ERC AdG no.~247277. Stig Larsson was supported by the Swedish Research
Council VR}

\date{\today}
\subjclass{35R60, 60H15, 35K70}

\begin{abstract}
  In this paper parabolic random partial differential equations and
  parabolic stochastic partial differential equations driven by a
  Wiener process are considered. A deterministic, tensorized evolution
  equation for the second moment and the covariance of the 
  solutions of the parabolic stochastic partial differential equations is derived.
  Well-posedness of a space-time weak variational
  formulation of this tensorized equation is established.
\end{abstract}

\maketitle
\section{Introduction}
\label{sec:intro}
With general results on existence and uniqueness of solutions 
of stochastic partial differential equations
being available (see, e.g., \cite{PZ07} and the references there),
the numerical solution of stochastic partial differential equations 
has received increasing attention in recent years.
The most widely used numerical solution approaches
are based on combinations of
time stepping, space discretization, and sample path simulation.
If the parabolic stochastic partial differential equation is linear and driven by $Q$-Wiener
noise, the linearity of the stochastic partial differential equation and of the mathematical
expectation can be exploited to show 
that the expectation of the square integrable 
random solution satisfies the corresponding deterministic, parabolic
evolution equation. 
In the present paper we show that
the covariance operator of the square integrable 
random solution of a parabolic stochastic partial differential equation satisfies
a deterministic, tensorized evolution equation
with a measure-valued, nonseparable forcing term.
We establish the well-posedness of this equation
in tensor products of Bochner spaces via a novel,
tensorized space-time variational formulation of this 
evolution equation.
This variational formulation, while being of interest in its 
own right, can serve as starting point for 
space-time compressive, adaptive Galerkin discretization
techniques as outlined in \cite{ScSt08,ScSt09,CS11}.

The outline of this paper is as follows: In Section~\ref{sec:PDEs}
we introduce the required notation, and recapitulate
basic results which are needed in the sequel. 
We review, in particular, the space-time variational formulation
of linear, parabolic evolution problems from~\cite{SS11}, and  
show how this can be tensorized by taking the tensor product of two such problems.
In Section~\ref{sec:SPDEs} we review the theory of stochastic partial differential equations of It\^o type as stochastic differential equations in Hilbert spaces as presented in~\cite{DPZ92}.
In Section~\ref{sec:4} we first present a tensorized, linear evolution equation for the second moment of the solution of a random evolution partial differential equation. We then state and prove the main result that 
the covariance operator of the mild solution of the stochastic
parabolic partial differential equation driven by $Q$-Wiener noise can be obtained 
as a weak, variational solution of a tensorized, 
deterministic parabolic evolution problem.
\section{Variational formulation of tensorized partial differential equations}
\label{sec:PDEs}
Let us review weak variational formulations of partial differential equations 
and solutions of tensorized equations in this section.
Accordingly, we let $(H, \langle \cdot,\cdot \rangle_H)$ denote 
a separable real Hilbert space and $A: \cD(A) \subset H \rightarrow H$ be a 
linear operator, which we assume 
to be  densely defined, self-adjoint, positive definite, 
and not necessarily bounded but with compact inverse. 
Then there exists an increasing sequence of real numbers 
$(\ga_k, k \in \N)$, which tends to infinity, and an orthonormal basis $(e_k, k \in \N)$ of~$H$ such that 
$A e_k = \ga_k e_k$. The domain of~$A$ is characterized by
  \begin{equation*}
   \cD(A)
    := \Bigl\{ \phi \in H, \,
	  \sum_{k=1}^\infty \ga_k^2 \langle \phi, e_k \rangle_H^2
	    < + \infty
      \Bigr\}.
  \end{equation*}
Furthermore, $-A$~is the generator of an analytic semigroup of contractions~$S = (S(t), t \in \R_+)$ and we are able to define the square root of the operator~$A$. This operator $A^{1/2}: \cD(A^{1/2}) \rightarrow H$ is given by
  \begin{equation*}
   A^{1/2} \phi 
    := \sum_{k=1}^\infty \ga_k^{1/2} \langle \phi, e_k\rangle_H \, e_k,
  \end{equation*}
for all $\phi \in \cD(A^{1/2})$, where
  \begin{equation*}
   \cD(A^{1/2})
      := \Bigl\{ \phi \in  H, \,
	    \|\phi\|_V^2
	    := \sum_{k=1}^\infty \ga_k \langle \phi, e_k \rangle_H^2 < + \infty
	\Bigr\},
  \end{equation*}
and $V := \cD(A^{1/2})$ together with the norm $\|\cdot \|_{V}$ becomes a Hilbert space. The norm satisfies that $\|\phi\|_{V} = \|A^{1/2} \phi\|_H$, for all $\phi \in V$.
Let us define the bilinear form $a:V \times V \rightarrow \R$ by
  \begin{equation*}
   a(\phi, \psi)
    := \langle A^{1/2} \phi, A^{1/2} \psi \rangle_H,
  \end{equation*}
for all $\phi, \psi \in V$. It is symmetric, continuous, coercive, and injective.
In the following, let $V^*$ denote the dual of $V$. By the Riesz representation theorem we identify $H$ with its dual and have the Gelfand triplet $V \subset H \cong H^* \subset V^*$. Then the linear operator $A:\cD(A) \subset H \rightarrow H$ can be interpreted as a bounded linear operator $A: V \rightarrow V^*$, $A \in L(V;V^*)$, via the bilinear form 
\begin{equation*}
 \dualpair{V^*}{V}{A\phi}{\psi} 
  = a(\phi,\psi)
  = \dualpair{V}{V^*}{\phi}{A^*\psi},
\end{equation*}
for $\phi,\psi\in V$, where $\dualpair{V^*}{V}{\cdot}{\cdot}$ denotes the dual pairing between $V$ and~$V^*$. Note that, although we assume that $A=A^*$ is self-adjoint, here and below
we write $A^*$ when the operator appears as an adjoint operator.
 The operator $A \in L(V;V^*)$ is boundedly invertible by the properties of~$a$ and the Lax--Milgram lemma, and its norm is bounded by
\begin{equation*}
 \|A\|_{L(V;V^*)}
  \le 1.
\end{equation*}
Let us fix the time interval $\IT:= [0,T]$, for some $T < + \infty$, and define the Hilbert spaces
  \begin{equation*}
   \cX := L^2(\IT;V)
    \quad \text{and} \quad
   \cY := L^2(\IT;V) \cap H^1_{0,\{T\}}(\IT;V^*),
  \end{equation*}
where
  \begin{equation*}
   H^1_{0,\{T\}}(\IT;V^*)
    := \{\phi \in H^1(\IT;V^*), \, \phi(T) = 0 \}.
  \end{equation*}
Let $\cX^*$ and $\cY^*$ denote the adjoint spaces with respect to the pivot space~$L^2(\IT;H)$, i.e., 
  \begin{equation*}
   \cX^*
    = L^2(\IT;V^*)
    \quad \text{and} \quad
   \cY^*
    = L^2(\IT;V^*) + H^{-1}(\IT;V).
  \end{equation*}
In this framework, let us consider the parabolic partial differential equation
  \begin{equation}\label{eq:PDE}
   (\partial_t + A) u
    = f
  \end{equation}
with initial condition $u(0) = u_0 \in H$, i.e., we want to solve the \emph{weak variational problem} to find, for given $u_0 \in H$ and $f \in \cY^*$, an element $u \in \cX$ such that for all $v \in \cY$
  \begin{equation}\label{eq:weak_var_problem}
   \cB(u,v)
    = \dualpair{\cY^*}{\cY}{f}{v} + \langle u_0, v(0) \rangle_H,
  \end{equation}
where
  \begin{equation*}
   \cB(u,v)
    := \int_0^T (\dualpair{V}{V^*}{u(t)}{-\partial_t v(t)} + a(u(t),v(t))) \, \dd t
    = \int_0^T \dualpair{V}{V^*}{u(t)}{(-\partial_t + A^*)v(t)} \, \dd t.
  \end{equation*}
It is shown in Theorem~2.6 in~\cite{SS11} that the variational problem~\eqref{eq:weak_var_problem} admits a unique solution $u \in \cX$ and that the operator $B:= \partial_t + A: \cX \rightarrow \cY^*$ is an isomorphism.

Let us consider the tensor spaces
  \begin{equation*}
   \cX^{(2)} 
    := \cX \otimes \cX
    \cong L^2(\IT;\R)^{(2)} \tilde{\otimes} V^{(2)}
  \end{equation*}
and
  \begin{equation*}
   \cY^{(2)}
    := \cY \otimes \cY.
  \end{equation*}
Since $(\cX^{(2)})^* = (\cX^*)^{(2)}$, we use the abbreviation $\cX^{(2)*}$ and we similarly define $\cY^{(2)*}$, which can be rewritten as
  \begin{align*}
    \cY^{(2)*}
      & = (L^2(\IT;V^*) + H^{-1}(\IT;V))^{(2)}\\
      & \cong L^2(\IT;\R)^{(2)} \tilde{\otimes} V^{(2)*}
	  + H^{-1}(\IT;\R)^{(2)} \tilde{\otimes} V^{(2)}\\
	  & \qquad + L^2(\IT;V^*) \otimes H^{-1}(\IT;V)
	  +  H^{-1}(\IT;V) \otimes L^2(\IT;V^*).
  \end{align*}
Here, we denote by $\tilde{\otimes}$ the tensor product which separates the spaces with respect to time and space.
Let us define the tensorized bilinear form
  \begin{align*}
   \cB^{(2)}(u,v)
    &:= \int_0^T \int_0^T \dualpair{V^{(2)}}{V^{(2)*}}{u(t,t')}{(-\partial_t + A^*)^{(2)} v(t,t')} \, \dd t \, \dd t'\\
    &= \dualpair{\cX^{(2)}}{\cX^{(2)*}}{u}{(-\partial_t + A^*)^{(2)} v},
  \end{align*}
for all $u \in \cX^{(2)}$ and $v \in \cY^{(2)}$. Then since $B = \partial_t + A$ is an isomorphism so is $B^{(2)}$, and the weak variational problem to find, for given $u_0 \in H^{(2)}$ and $f \in \cY^{(2)*}$, an element $u \in \cX^{(2)}$ such that
  \begin{equation}\label{eq:weakWienerEqn}
   \cB^{(2)}(u,v)
    = \dualpair{\cY^{(2)*}}{\cY^{(2)}}{f}{v} + \langle u_0, v(0,0) \rangle_{H^{(2)}},
  \end{equation}
for all $v \in \cY^{(2)}$, admits a unique solution. The corresponding strong form of the tensorized partial differential equation~\eqref{eq:weakWienerEqn} reads
  \begin{align*}
   (\partial_t + A)^{(2)} u & = f,\\
   (\partial_t + A) \otimes I u(\cdot,0) & = 0,\\
   I \otimes (\partial_t + A) u(0,\cdot) & = 0,\\
   u(0,0) & = u_0.
  \end{align*}

In Section~\ref{sec:4} we will show that second moments and covariances of stochastic and random partial differential equations are solutions of the tensorized partial differential equation~\eqref{eq:weakWienerEqn} with $u_0$ and $f$ chosen appropriately.

\section{Stochastic partial differential equations}\label{sec:SPDEs}

In this section, we consider stochastic partial differential equations and their mild solutions in the framework of~\cite{DPZ92} and~\cite{PZ07}.

Let $(\gO,\cA,(\cF_t)_{t\ge 0},P)$ be a filtered probability space that
satisfies the ``usual conditions''. Furthermore, let $(H,\langle
\cdot,\cdot\rangle_H)$ be a separable real Hilbert space with corresponding norm denoted by~$\|\cdot\|_H$. The \emph{Lebesgue--Bochner space $L^2(\gO;H)$} is the space of all square integrable, $H$-valued random variables, i.e., the space of all $H$-valued random variables~$X$ such that
  \begin{equation*}
   \|X\|_{L^2(\gO;H)}^2
    := \E[ \|X\|_H^2]
    < + \infty.
  \end{equation*}
For all $X \in L^2(\gO;H)$, the \emph{second moment of~$X$}
  \begin{equation*}
   \IM^{(2)} X
    := \E[ X^{(2)}]
    = \E[X \otimes X]
  \end{equation*}
is well-defined as an element of $H^{(2)}$ since
  \begin{equation*}
   \|\IM^{(2)} X \|_{H^{(2)}}
    \le \E[ \|X^{(2)}\|_{H^{(2)}}]
    = \E[ \|X\|_H^2]
    = \|X\|_{L^2(\gO;H)}^2.
  \end{equation*}
Furthermore, we define the \emph{covariance~$\Cov(X)$ of $X \in L^2(\gO;H)$} by
  \begin{equation*}
   \Cov(X)
    := \IM^{(2)} (X - \E[X]),
  \end{equation*}
i.e., $\Cov(X)$ is the centered second moment of~$X$.

Let us denote by $\LNP(H)$ the space of all nonnegative, symmetric, nuclear operators on~$H$. Then there exists a unique operator $Q \in \LNP(H)$ such that
  \begin{equation*}\label{eq:rel-cov-Q}
   \langle \Cov(X), \varphi \otimes \psi \rangle_{H^{(2)}}
    = \langle Q \varphi, \psi \rangle_H,
  \end{equation*}
for all $\varphi, \psi \in H$. The operator~$Q$ is called the \emph{covariance operator of~$X$}.

For $Q \in \LNP(H)$, the $H$-valued stochastic process $W := (W(t), t \in \R_+)$ is called a \emph{$Q$-Wiener process} if it starts in zero $P$-almost surely, it has $P$-almost surely continuous trajectories, the increments are independent, and for $s<t$, the increment $W(t) - W(s)$ is Gaussian distributed with expectation zero and covariance operator $(t-s)Q$. Let us denote by $q$ the covariance in~$H^{(2)}$ that corresponds to~$Q$, i.e., $q$ is uniquely defined by the condition
 \begin{equation*}
   \langle q, \varphi \otimes \psi \rangle_{H^{(2)}}
    = \langle Q \varphi, \psi \rangle_H,
 \end{equation*}
for all $\varphi,\psi \in H$.
Since $Q\in \LNP(H)$, there exists an orthonormal basis
$(e_n,n \in \N)$ of $H$ consisting of eigenvectors of~$Q$. Therefore, we
have the representation $Q e_n = \gg_n e_n$, 
where $\gamma_n \ge 0$ is the eigenvalue corresponding to~$e_n$, for $n \in \N$. 
Then the square root of $Q$ is defined as
\begin{equation*}
 Q^{1/2}\psi 
  := \sum_{n \in \N} \gamma_n^{1/2} \langle \psi,e_n \rangle_H \,  e_n,
\end{equation*}
for $\psi \in H$,
and $Q^{-1/2}$ denotes the pseudo inverse of~$Q^{1/2}$. 
Let us denote by
$(\cH,\langle\cdot,\cdot\rangle_{\cH})$ the Hilbert space defined by 
$\cH := Q^{1/2}(H)$ 
endowed with the inner product
$\langle\psi,\phi\rangle_{\cH} := \langle Q^{-1/2}\psi,Q^{-1/2}\phi\rangle_H$, for $\psi,\phi \in \cH$.
Let $\LHS(\cH;H)$ refer to the space of all Hilbert--Schmidt
operators from $\cH$ to~$H$, and 
by $\| \cdot \|_{\LHS(\cH;H)}$ we denote the corresponding norm.
It holds that $q$ has an expansion with respect to the eigenvalues and eigenvectors of~$Q$ given by
  \begin{equation}\label{eq:ev_expansion_q}
   q = \sum_{n \in \N} \gg_n \, e_n \otimes e_n.
  \end{equation}
Furthermore, $W$ admits a Karhunen--Lo\`eve expansion, i.e.,
for all $t \in \IT$, it holds that
\begin{equation*}
W(t) = \sum_{n\in\N} \langle W(t),e_n\rangle_H \,e_n
      = \sum_{n\in\N} \gamma_n^{1/2} \, W_n(t) \, e_n,
\end{equation*}
where $(W_n, n \in \N)$ is a sequence of independent, real-valued Brownian motions.  

For $t \in \IT$, let us denote by
  \begin{equation*}
   \int_0^t \Phi(s) \, \dd W(s)
  \end{equation*}
the stochastic integral with respect to $\Phi \in \IL^2_{\cH,\IT}(H)$ with
\begin{equation*}
 \E \left[ \int_0^T \|\Phi(s)\|_{\LHS(\cH;H)}^2 \, \dd s\right]
   < + \infty .
\end{equation*}
Here, $\IL^2_{\cH,\IT}(H) := L^2(\gO \times \IT; \LHS(\cH;H))$ denotes
the space of integrands with respect to the measure space $(\Omega \times \IT, \cP_{\IT}, \, P
\otimes \lambda)$, where
$\cP_{\IT}$ is the $\sigma$-algebra of predictable sets in
$\Omega \times \IT$ and $\lambda$ denotes the Lebesgue measure.

Then the stochastic integral satisfies the It\^o isometry
  \begin{equation}\label{eq:Ito_isometry}
   \E \left[ \bigl\| \int_0^t \Phi(s) \, \dd W(s) \bigr\|_H^2 \right]
    = \E \left[ \int_0^t \|\Phi(s)\|_{\LHS(\cH;H)}^2 \, \dd s \right],
  \end{equation}
for all $t \in \IT$.
For an introduction to stochastic integrals with respect to Hilbert-space-valued stochastic processes, the reader is referred to~\cite{C07, DPZ92, PZ07, PR07}.

We define the \emph{weak stochastic integral} by
\begin{equation*}
 \int_0^t \langle \Psi(s), \Phi(s) \, \dd W(s) \rangle_H
    := \int_0^t \tilde{\Phi}_\Psi(s) \, \dd W(s),
\end{equation*}
for $t \in \IT$, with $\tilde{\Phi}_\Psi(s)\in\cH^*$, $s \in \IT$, such that for all $u \in \cH$
\begin{align*}
 \tilde{\Phi}_\Psi(s) u 
    := \langle \Psi(s), \Phi(s) u \rangle_H ,
\end{align*}
where $\Psi$ is an $H$-valued, continuous, adapted
stochastic process and $\Phi \in \IL^2_{\cH,\IT}(H)$.  The weak stochastic
integral is well-defined (cf.~\cite{PZ07}). 
%
Since the stochastic integral is a martingale (cf.\ Theorem~4.12 in~\cite{DPZ92}), it satisfies for $t \in \IT$ that
\begin{equation}\label{eq:stoch_int_cond_expect}
 \E\left[\left.
      \int_0^t \langle \Psi(s), \Phi(s) \, \dd W(s)\rangle_H
  \right| \cF_{0} \right]
    = \E\left[\left.
      \int_0^t \tilde{\Phi}_\Psi(s) \, \dd W(s)
  \right| \cF_{0} \right]
    = 0 .
\end{equation}
Furthermore, the weak stochastic integral satisfies a general It\^o isometry, which is shown in the following lemma.

\begin{lemma}\label{lem:Ito}
 Let $v_1, v_2 \in C^0(\IT; H)$ and $\Phi \in \IL^2_{\cH,\IT}(H)$. Then the weak stochastic integral satisfies
  \begin{align*}
   \E\Bigl[ \int_0^t \langle v_1(s), \Phi(s) \, \dd W(s)\rangle_H & \int_0^t \langle v_2(s'), \Phi(s') \, \dd W(s')\rangle_H \Bigr]\\
    &= \int_0^t 
	\langle v_1(s) \otimes v_2(s), \E[\Phi(s) \otimes \Phi(s)] q\rangle_{H^{(2)}} \, \dd s .
  \end{align*}
\end{lemma}

\begin{proof} 
 Let $t \in \IT$ be fixed.
 Using the Karhunen--Lo\`eve expansion of~$W$, we rewrite the weak stochastic integral
\begin{align*}
 \int_0^t \langle v_1(s), \Phi(s) \, \dd W(s)\rangle_H
=
  \sum_{n=1}^\infty {\gamma_n}^{1/2} \int_0^t
	    \langle v_1(s), \Phi(s) e_n \rangle_H \, \dd W_n(s) .
\end{align*}
Then the independence as well as the martingale property of the real-valued Wiener processes~$W_n$, $n \in \N$, implies that
\begin{align*}
&  \E\biggl[
    \int_0^t \langle v_1(s), \Phi(s) \, \dd W(s) \rangle_H 
    \int_0^t \langle v_2(s'), \Phi(s')\,\dd W(s')\rangle_H
  \biggr]\\
& \qquad =
   \sum_{n,m=1}^\infty {\gamma_n^{1/2} \gamma_m^{1/2}} \,
      \E\left[
	\int_0^t \langle v_1(s),\Phi(s) e_n\rangle_H \, \dd W_n(s)
	\int_0^t \langle v_2(s'), \Phi(s') e_m \rangle_H \, \dd W_m(s')
      \right]\\
&\qquad =
   \sum_{n=1}^\infty \gamma_n \,
      \E\left[
	\int_0^t \langle v_1(s), \Phi(s) e_n\rangle_H \, \dd W_n(s)
	\int_0^t \langle v_2(s'), \Phi(s') e_n\rangle_H \, \dd W_n(s')
      \right]\\
& \qquad =  
   \sum_{n=1}^\infty \gamma_n \,
      \E\left[
	\int_0^t \langle v_1(s), \Phi(s) e_n\rangle_H
	      \langle v_2(s), \Phi(s) e_n\rangle_H \, \dd s
      \right],
\end{align*}
 where we applied an It\^o isometry for real-valued Wiener processes (see Proposition~1.2 in~\cite{IW89}) in the last step.
 By the definition of the tensor product and of~$q$ in~\eqref{eq:ev_expansion_q}, we get that
 \begin{align*}
   \sum_{n=1}^\infty \gamma_n \,
      \E\Bigl[
	\int_0^t & \langle v_1(s), \Phi(s) e_n\rangle_H
	      \langle v_2(s), \Phi(s) e_n\rangle_H \, \dd s
      \Bigr]\\
    & = \int_0^t \E \left[
	\Bigl\langle v_1(s) \otimes v_2(s), \sum_{n=1}^\infty \gg_n (\Phi(s) e_n) \otimes (\Phi(s) e_n)\Bigr\rangle_{H^{(2)}} \, \dd s
      \right]\\
    & = \int_0^t 
	\langle v_1(s) \otimes v_2(s), \E[\Phi(s) \otimes \Phi(s)] q\rangle_{H^{(2)}} \, \dd s ,
 \end{align*}
 which finishes the proof.
\end{proof}
%

Having introduced some properties of Hilbert-space-valued random variables and the stochastic integral with respect to a $Q$-Wiener process, we start the discussion of the stochastic partial differential equation
\begin{equation}\label{eq:Wiener-SPDE}
\dd X(t) + AX(t) \, \dd t = \dd W(t)
\end{equation}
with $\cF_0$-measurable initial condition $X(0) = X_0 \in L^2(\gO;H)$, for $t \in \IT$. Here, $W$ is a $Q$-Wiener process and $A$ satisfies the assumptions that were made in Section~\ref{sec:PDEs}, i.e., $A$ is a linear operator that is densely defined, self-adjoint, positive definite, and not necessarily bounded but with compact inverse. Then \eqref{eq:Wiener-SPDE} admits a unique \emph{mild solution~$X$}, i.e., $X$ is a predictable process that satisfies $\sup_{t \in \IT} \|X(t)\|_{L^2(\gO;H)} < + \infty$ and
  \begin{equation*}
   X(t)
    = S(t) X_0 + \int_0^t S(t-s) \, \dd W(s),
  \end{equation*}
for all $t \in \IT$. Furthermore, this solution is equal to the unique \emph{weak solution}, i.e., $X$ is a predictable process that satisfies $\sup_{t \in \IT} \|X(t)\|_{L^2(\gO;H)} < + \infty$ and for all $\varphi \in \cD(A^*)$ and $t \in \IT$
  \begin{equation*}
   \langle \varphi, X(t) \rangle_H
    = \langle \varphi, X_0 \rangle_H
      - \int_0^t \langle A^* \varphi, X(s) \rangle_H \, \dd s
      + \langle \varphi, W(t) \rangle_H. 
  \end{equation*}
Recall that, although $A=A^*$ is self-adjoint, 
we write $A^*$ when the operator appears as an adjoint operator.
The solution~$X$ of~\eqref{eq:Wiener-SPDE} has certain
properties. In the following lemma, we establish another weak
formulation. There we use the space of test functions 
  \begin{equation*}
    C^1_{0,\{T\}}(\IT;\cD(A^*)):= \{\phi \in C^1(\IT;\cD(A^*)), \, \phi(T) = 0 \}.
  \end{equation*}
\begin{lemma}\label{lem:weak_solution}
 Let $X$ be the mild solution of the stochastic partial differential equation~\eqref{eq:Wiener-SPDE}. Then $X$ satisfies for all $\varphi \in C^1_{0,\{T\}}(\IT;\cD(A^*))$ $P$-almost surely that
  \begin{equation*}
   \langle X, (-\partial_t + A^*) \varphi \rangle_{L^2(\IT;H)}
      = \langle X_0, \varphi(0)\rangle_H
	+ \int_0^T \langle \varphi(t), \dd W(t) \rangle_H.
  \end{equation*}
\end{lemma}

\begin{proof}
 The lemma is a direct consequence of (9.20) in~\cite{PZ07}, where we use that $\varphi(T) = 0$ by definition.
\end{proof}

Furthermore, we prove regularity of the second moment of the solution.

\begin{lemma}
 The second moment~$\IM^{(2)}X$ of the mild solution~$X$ of~\eqref{eq:Wiener-SPDE} is an element of~$\cX^{(2)}$.
\end{lemma}

\begin{proof}
 We first observe that
  \begin{equation*}
   \|\IM^{(2)} X\|_{\cX^{(2)}}
    \le \E[ \|X^{(2)}\|_{\cX^{(2)}}]
    = \E[ \|X\|_{\cX}^2].
  \end{equation*}
 Using the definition of the norm and the independence of the stochastic integral of~$\cF_0$, we obtain that
  \begin{align*}
   \E[ \|X\|_{\cX}^2]
    & = \int_0^T \E\left[\Big\|A^{1/2} S(t)X_0 + A^{1/2} \int_0^t S(t-s) \, \dd W(s) \Big\|_H^2 \right] \, \dd t\\
    & = \E\left[ \int_0^T \|A^{1/2} S(t)X_0\|_H^2 \, \dd t \right]
      + \int_0^T \E\left[ \Big\|\int_0^t A^{1/2} S(t-s) \, \dd W(s) \Big\|_H^2 \right] \, \dd t.
  \end{align*}
For the first term we use that for $v \in H$
  \begin{equation}  \label{eq:1}
    \int_0^T \|A^{1/2} S(t)v\|_H^2 \, \dd t 
    \le \tfrac12\|v\|_H^2, 
  \end{equation}
which is easily proved by the spectral representation of $S(t)$.
Hence, the first term is bounded by $\tfrac12 \|X_0\|_{L^2(\gO;H)}^2$.
To the second term we apply the It\^o isometry~\eqref{eq:Ito_isometry}, the definition of
the Hilbert--Schmidt norm, and again \eqref{eq:1} to get
\begin{align*}
 \E\left[\Big\|\int_0^t A^{1/2} S(t-s) \, \dd W(s) \Big\|_H^2 \right]
  & = \int_0^t \|A^{1/2} S(t-s)\|_{\LHS(\cH;H)}^2 \, \dd s\\
  & = \sum_{n \in \N} \gg_n \int_0^t \|A^{1/2}S(t-s) e_n\|_H^2 \, \dd s
  \le \tfrac12 \Tr(Q).
\end{align*}
This leads to
\begin{equation*}
   \|\IM^{(2)} X\|_{\cX^{(2)}}
    \le \tfrac12 \bigl( \|X_0\|_{L^2(\gO;H)}^2 + T \Tr(Q) \bigr)
    < + \infty
  \end{equation*}
by our assumptions on $X_0$ and $Q$. 
\end{proof}

\section{Covariance partial differential equation} \label{sec:4}

In this section we calculate second moments and covariances of
two different classes of stochastic evolution 
partial differential equations. First, we look at random partial
differential equations, where the initial condition and the right hand
side of~\eqref{eq:PDE} are random variables. In a second
step we deal with mild solutions of the parabolic
stochastic partial differential equation~\eqref{eq:Wiener-SPDE}. 
We show that second
moments and covariances of the equations are equal to the unique
solution of the tensorized partial differential
equation~\eqref{eq:weakWienerEqn}, where the parameters $u_0$ and $f$
in~\eqref{eq:weakWienerEqn} have to be chosen according to
the random or stochastic partial differential equation.

\subsection{Random partial differential equation}
\label{ssec:RndPDE}
In the framework of Section~\ref{sec:PDEs}, 
let us consider in a first step the random partial differential equation 
\begin{equation}\label{eq:RPDE}
 (\partial_t + A) U = F
\end{equation}
on the time interval~$\IT$ with initial condition 
$U(0) = U_0 \in L^2(\gO;H)$ and $F \in L^2(\gO;\cY^*)$. 
For almost every fixed $\go \in \gO$,  it holds that $U_0(\go) \in H$ 
and $F(\go) \in \cY^*$ and it was shown in Section~\ref{sec:PDEs} that the 
weak variational problem
  \begin{equation}\label{eq:RPDE_variational}
   \cB(u,v) = \dualpair{\cY^*}{\cY}{F(\go)}{v} + \langle U_0(\go), v(0) \rangle_H
  \end{equation}
has a unique solution~$u \in \cX$ that depends on~$\go$ and that we denote by $U(\go) := u$. Furthermore, since $\partial_t + A$ is an isomorphism with bounded inverse, $U$ is an $\cX$-valued random variable that is in~$L^2(\gO;\cX)$ by the properties of $U_0$ and~$F$. Therefore, $U$ is a solution of~\eqref{eq:RPDE} in the sense that $U \in L^2(\gO;\cX)$ and $U(\go)$ satisfies the weak variational problem~\eqref{eq:RPDE_variational} for almost every $\go \in \gO$.

Since $U \in L^2(\gO;\cX)$, the second moment $\IM^{(2)} U$ of the variational solution~$U$ of the random partial differential equation~\eqref{eq:RPDE} is well-defined. Furthermore, $\IM^{(2)} U$ satisfies for $v_1, v_2 \in \cY$ that
  \begin{align*}
   \cB^{(2)}(\IM^{(2)} U , v_1 \otimes v_2)
    & = \E[ \cB(U, v_1) \cB(U,v_2) ]\\
    & = \E[ (\dualpair{\cY^*}{\cY}{F}{v_1} + \langle U_0, v_1(0) \rangle_H)
	  (\dualpair{\cY^*}{\cY}{F}{v_2} + \langle U_0, v_2(0) \rangle_H) ]\\
    & = \dualpair{\cY^{(2)*}}{\cY^{(2)}}{\IM^{(2)}F}{v_1 \otimes v_2}
	+ \langle \IM^{(2)} U_0, (v_1 \otimes v_2)(0,0) \rangle_{H^{(2)}}\\
	& \qquad + \E [ \dualpair{\cY^*}{\cY}{F}{v_1}\, \langle U_0, v_2(0) \rangle_H]
	+ \E [ \langle U_0, v_1(0) \rangle_H\, \dualpair{\cY^*}{\cY}{F}{v_2} ].
  \end{align*}
The two summands in the last line will be equal to zero if we assume that $F$ and $U_0$ are independent and that $\E(F) = 0$ or $\E(U_0) = 0$. This implies that $\IM^{(2)} U$ solves the weak variational problem~\eqref{eq:weakWienerEqn} with $u_0 := \IM^{(2)} U_0$ and $f := \IM^{(2)} F$, if $F$ and $U_0$ are independent and if $\E(F) = 0$ or $\E(U_0) = 0$.

Let us next look at the covariance $\Cov(U) = \IM^{(2)}(U - \E[U])$ of the solution~$U$ of~\eqref{eq:RPDE}. It satisfies for $v_1, v_2 \in \cY$ that
  \begin{align*}
      \cB^{(2)}(\Cov(U) , v_1 \otimes v_2)
    & = \E[ \cB(U - \E[U], v_1) \cB(U - \E[U],v_2) ]\\
    & = \dualpair{\cY^{(2)*}}{\cY^{(2)}}{\Cov(F)}{v_1 \otimes v_2}
	+ \langle \Cov(U_0), (v_1 \otimes v_2)(0,0) \rangle_{H^{(2)}}\\
	& \qquad + \E [ \dualpair{\cY^*}{\cY}{F - \E[F]}{v_1} \, \langle U_0 - \E[U_0], v_2(0) \rangle_H]\\
	& \qquad + \E [ \langle U_0  - \E[U_0], v_1(0) \rangle_H \, \dualpair{\cY^*}{\cY}{F - \E[F]}{v_2} ],
  \end{align*}
since
  \begin{equation*}
   \cB(\E[U], v)
      = \E[ \cB(U,v) ]
      = \E[\dualpair{\cY^*}{\cY}{F}{v} + \langle U_0, v(0) \rangle_H]
      = \dualpair{\cY^*}{\cY}{\E[F]}{v} + \langle \E[U_0], v(0) \rangle_H,
  \end{equation*}
for all $v \in \cY$. Therefore, $\Cov(U)$ solves the 
weak variational problem~\eqref{eq:weakWienerEqn} with 
$u_0 := \Cov(U_0)$ and $f := \Cov(F)$, 
provided that $F$ and $U_0$ are independent.

  We have thus seen that, under additional assumptions concerning
  the independence and the expectation of the data $U_0$ and~$F$, the second moment and
  the covariance of the random partial differential
  equation~\eqref{eq:RPDE} are the unique solutions of the tensorized
  partial differential equation~\eqref{eq:weakWienerEqn} with
  appropriately chosen data $u_0$
  and $f$. We will see in the next section that
  such assumptions are naturally met by the data of the
  stochastic partial differential equation~\eqref{eq:Wiener-SPDE}.

\subsection{Stochastic partial differential equation}
\label{ssec:SPDE}
In this subsection we consider the second moment~$\IM^{(2)} X$ of the
mild solution~$X$ of the stochastic partial differential
equation~\eqref{eq:Wiener-SPDE} with initial condition $X_0 \in
L^2(\gO;H)$ and show that it is equal to the unique solution of the
weak variational problem~\eqref{eq:weakWienerEqn}, where $u_0 =
\IM^{(2)}X_0$ and $f = \delta \tilde{\otimes} q$. 
Here, $\delta$ is a Dirac distribution~$\delta$ to be introduced next.

To describe the temporal correlation of the $Q$-Wiener 
process~$W$, 
we define the distribution~$\delta$ in the sense of 
L.~Schwartz as functional acting on
test functions $\phi\in C_0^\infty(\IT^2;\R)$ by 
\begin{align*}
  \langle \delta,\phi\rangle_{L^2(\IT;\R)^{(2)}} 
      = \int_0^T \phi(s,s)\,\dd s.
\end{align*}
Therefore, 
$\delta$ is a measure which is
supported on the diagonal of $\IT^2$. 
By Lemma~3 in~\cite{ST03},
$\delta\in H^{-s,-s}(\IT^2;\R) \cong H^{-s}(\IT;\R)^{(2)}$, for $s >
1/4$.  
In the following lemma we prove that an 
additional spatial regularity assumption on the covariance operator~$Q$
combined with the low temporal regularity of the Wiener process
implies that the tensor product $\delta \tilde{\otimes} q$ is in
$H^{-1}(\IT;V)^{(2)}$ and therefore in~$\cY^{(2)*}$.

\begin{lemma}\label{lem:prop_delta_q}
 Let $\Tr(AQ)<+\infty$. Then $\delta \tilde{\otimes} q \in \cY^{(2)*}$.
\end{lemma}

\begin{proof}
 We first remark that $\Tr(AQ)<+\infty$ implies by~\eqref{eq:ev_expansion_q} that
  \begin{equation*}
   \|q\|_{V^{(2)}}
    = \Bigl\| \sum_{n=1}^\infty \gg_n e_n \otimes e_n \Bigr\|_{V^{(2)}}
    \le \sum_{n=1}^\infty \gg_n \|e_n \otimes e_n \|_{V^{(2)}}
    = \Tr(AQ)
    < + \infty,
  \end{equation*}
  and hence that $q \in V^{(2)}$.  It remains to show that $\delta \in
  H^{-1}(\IT;\R)^{(2)}$ to finish the proof. This is true since
  $\delta \in H^{-s}(\IT;\R)^{(2)}$, for all $s > 1/4$ by Lemma~3
  in~\cite{ST03}.
\end{proof}

Having shown some regularity of the expression $\delta \tilde{\otimes} q$, we are now able to state and prove the main result of the paper.

\begin{theorem}\label{thm:QWiener}
  Let $X$ be the mild solution of the stochastic partial differential
  equation~\eqref{eq:Wiener-SPDE} with $\cF_0$-measurable initial
  condition $X_0 \in L^2(\gO;H)$. Moreover, assume that
  $\Tr(AQ)<+\infty$.  Then the second moment~$\IM^2 X$ solves
  the weak variational problem~\eqref{eq:weakWienerEqn} with $u_0 :=
  \IM^{(2)} X_0$ and $f := \delta \tilde{\otimes} q$.
\end{theorem}

\begin{proof}
 First, we remark that the embedding $C^1_{0,\{T\}}(\IT;\cD(A^*)) \subset \cY$ is continuous and dense. Therefore, it is sufficient to show \eqref{eq:weakWienerEqn} for $v_1 \otimes v_2$ with $v_1, v_2 \in C^1_{0,\{T\}}(\IT;\cD(A^*))$.
 So, let $v_1,v_2 \in C^1_{0,\{T\}}(\IT;\cD(A^*))$. Then using the definition of the bilinear form~$\cB^{(2)}$, we obtain that
  \begin{align*}
   \cB^{(2)}(\IM^2 X,v_1 \otimes v_2)
    & = \dualpair{\cX^{(2)}}{\cX^{(2)*}}{\IM^2 X}{(-\partial_t + A^*)^{(2)}(v_1 \otimes v_2)}\\
    & = \E[ \dualpair{\cX}{\cX^*}{X}{(-\partial_t + A^*) v_1}\, \dualpair{\cX}{\cX^*}{X}{(-\partial_t + A^*) v_2}].
  \end{align*}
 Furthermore, the regularity of $v_1$ and $v_2$ implies that
  \begin{align*}
   \E[ \dualpair{\cX}{\cX^*}{X}{(-\partial_t + A^*) v_1}\, &\dualpair{\cX}{\cX^*}{X}{(-\partial_t + A^*) v_2}]\\
    & = \E[ \langle X, (-\partial_t + A^*) v_1 \rangle_{L^2(\IT;H)} \langle X, (-\partial_t + A^*) v_2 \rangle_{L^2(\IT;H)} ].
  \end{align*}
 The application of Lemma~\ref{lem:weak_solution} leads to
  \begin{align*}
   \E[ &\langle X, (-\partial_t + A^*) v_1 \rangle_{L^2(\IT;H)} \langle X, (-\partial_t + A^*) v_2 \rangle_{L^2(\IT;H)}]\\
     & = \E[ \langle X_0, v_1(0) \rangle_H \langle X_0, v_2(0) \rangle_H]
	+ \E\left[ \langle X_0, v_1(0) \rangle_H \int_0^T \langle v_2(t), \dd W(t)\rangle_H \right]\\
	& \quad + \E\left[ \int_0^T \langle v_1(t), \dd W(t) \rangle_H \langle X_0, v_2(0) \rangle_H\right]
	+ \E\left[ \int_0^T \langle v_1(t), \dd W(t)\rangle_H \int_0^T \langle v_2(t), \dd W(t)\rangle_H \right].
  \end{align*}
 The second and third terms on the right hand side are equal to zero by~\eqref{eq:stoch_int_cond_expect} and since $X_0$ is $\cF_0$-measurable. We apply Lemma~\ref{lem:Ito} to the fourth term and the definition of the tensor product to the first and the fourth terms to derive
  \begin{align*}
   \E[ \langle X, (-\partial_t + A^*) v_1 \rangle_{L^2(\IT;H)} &\langle X, (-\partial_t + A^*) v_2 \rangle_{L^2(\IT;H)} ]\\
    & = \langle \IM^2 X_0, (v_1 \otimes v_2)(0,0) \rangle_{H^{(2)}}
      + \langle \delta \tilde{\otimes} q, v_1 \otimes v_2 \rangle_{L^2(\IT;H)^{(2)}}.
  \end{align*}
 The desired assertion now follows, since
  \begin{equation*}
   \langle \delta \tilde{\otimes} q, v_1 \otimes v_2 \rangle_{L^2(\IT;H)^{(2)}}
    = \dualpair{\cY^{(2)*}}{\cY^{(2)}}{\delta \tilde{\otimes} q}{v_1 \otimes v_2},
  \end{equation*}
 due to Lemma~\ref{lem:prop_delta_q} and the properties of $v_1$ and~$v_2$.
\end{proof}

\begin{corollary}
 Under the assumptions of Theorem~\ref{thm:QWiener},
 the covariance $\Cov(X)$ of the mild solution~$X$ of~\eqref{eq:Wiener-SPDE} satisfies
  \begin{equation*}
   \cB^{(2)}(\Cov(X),v) = \dualpair{\cY^{(2)*}}{\cY^{(2)}}{\delta \tilde{\otimes} q}{v} + \langle \Cov(X_0), v(0,0)\rangle_{H^{(2)}},
  \end{equation*}
for all $v \in \cY^{(2)}$, and therefore it solves the weak variational problem~\eqref{eq:weakWienerEqn} with $u_0 := \Cov(X_0)$ and $f:= \delta \tilde{\otimes} q$.
\end{corollary}

\begin{proof}
 For all $t \in \IT$, it holds that
  \begin{align*}
   X(t) - \E[X(t)]
    & = S(t) X_0 + \int_0^t S(t-s) \, \dd W(s) - S(t) \E[X_0]\\
    & = S(t) (X_0 - \E[X_0]) + \int_0^t S(t-s) \, \dd W(s),
  \end{align*}
i.e., if $X$ is a mild solution of~\eqref{eq:Wiener-SPDE} with initial condition~$X_0$, then $X - \E[X]$ is mild solution of~\eqref{eq:Wiener-SPDE} with initial condition~$X_0 - \E[X_0]$. The assertion follows by an application of Theorem~\ref{thm:QWiener} to the transformed equation.
\end{proof}


\end{document}